\documentclass{amsart}
\usepackage{graphicx} 
\usepackage{amsmath, amssymb, amsfonts, amsthm, thmtools}
\usepackage[dvipsnames]{xcolor}
\usepackage{tikz}
\usetikzlibrary{arrows.meta,positioning}
\usepackage{bbm}

\title{The Single Ring Theorem and a Question of Shub}
\author{Joshua Paik} 
\date{May 2025}

\usepackage{biblatex} 
\definecolor{muuve}{rgb}{0.467, 0.3,0.52}

\usepackage[colorlinks=true,
            bookmarksopen=true,
            linkcolor=muuve,
            urlcolor=muuve,
            citecolor=muuve]{hyperref}
\usepackage[capitalise]{cleveref} 

\declaretheorem[numberwithin=section]{theorem} 
 
\declaretheorem[sibling=theorem]{corollary}

\declaretheorem[sibling=theorem]{conjecture}

 \newcommand{\cB}{\mathcal{B}} 
            \newcommand{\cP}{\mathcal{P}}          

\renewcommand{\epsilon}{\varepsilon}

\newcommand{\C}{\mathbb{C}}

\newcommand{\E}{\mathbb{E}}

\newcommand{\N}{\mathbb{N}}
\renewcommand{\O}{\mathrm{SO}}
\renewcommand{\P}{\mathbb{P}}

\newcommand{\R}{\mathbb{R}}

\newcommand{\U}{\mathrm{SU}}

\newcommand{\SO}{\mathrm{SO}}
\newcommand{\SL}{\mathrm{SL}}

\begin{document}


\begin{abstract}
    Given an orthogonally invariant probability measure on $GL(d,\R)$, Mike Shub asked whether the average product of the $k$ top eigenvalues in the ensemble can be lower bounded by the average distortion along $k$ dimensional Grassmanians. Recently, \cite{armentanorandom} provided partial progress, however they attach a constant $c_{d,k} \to 0$ as $d \to \infty$.  In this paper, by invoking the Single Ring Theorem \cite{guionnet2011single} \cite{guionnet2012support}, we show the conjecture asymptotically for the spectral radius, in particular, $c_{d,k} \to 1$ as $d \to \infty$, and $k = 1$. 
\end{abstract}

\maketitle

\section{Introduction}
The following is a question of Mike Shub.\footnote{Over $GL(d,\R)$, this is a question that appears (as a comment) in Dedieu--Shub \cite{dedieu2003mike} and a survey article (as a question) of Burns, Pugh, Shub, Wilkinson \cite{burns2001recent}.} Recall that a uniform random $k$--dimensional Grassmanian in $\C^d$, denoted $g_{k}$, can be represented as the span of the first $k$--columns of a Haar random $U \in \U(d)$ or $\O(d)$, and we denote this random $d \times k$ rectangular matrix as $U_k$. Then for a square matrix $A$, define $\det A|g_k := \det (AU_k)^*(AU_k)$. 

\begin{conjecture}[\cite{dedieu2003mike},\cite{burns2001recent}] \label{conjecture}
     Let $A \in GL(d,\R) \setminus \R I$. Then
     \begin{enumerate}
         \item $$\int_{\SO(d)} \prod \limits_{i=1}^k |\lambda_i (UA )| \ d\nu(U) \geq c_{d,k} \int_{\mathrm{Grass}(d,k)} \det A|g_{k} \ d\mu(g_{k}),$$
         \item $$\int_{\SO(d)} \sum \limits_{i=1}^k \log |\lambda_i (UA )| \ d\nu(U) \geq c_{d,k} \int_{\mathrm{Grass}(d,k)} \log \det A|g_{k} \ d\mu(g_{k})\text{ and,}$$
         \item One can choose $c_{d,k} = 1$ for all $d$ and $k$. 
     \end{enumerate}
\end{conjecture}

All experimental evidence suggests that $c_{d,k} = 1$. Over $\mathrm{SU}(d)$ the full conjecture was proven by Dedieu-Shub \cite{dedieu2003mike}. Recently, Armentano, Chinta, Sahi and Shub \cite{armentanorandom} recently proved the following result towards proving Conjecture \ref{conjecture}. 

\begin{theorem}[
 \cite{armentanorandom}]
Let $G = \SL(d,\R)$,$K = \SO(d)$, and $X = \mathrm{Grass}(d,k)$, the space of real $k$--dimensional Grassmanians of $\R^d$. Then for every $k \in [1,..,d]$ we have that 

$$\int_{SO(d)} \sum_{i=1}^k \log |\lambda_i(OA)| d\nu(O) \geq c_{d,k} \int_X \log \det A | g_{k} \ d\mu(g_{k})$$

 where $c_{d,k} = \frac{1}{\binom{d}{k}}$.
\end{theorem}

When $k=1$, Rivin \cite{rivin2005CMP} also proved an inequality with a weaker constant. In particular, for both \cite{armentanorandom} and \cite{rivin2005CMP}, for every $k$, as $d$ increases, their $c_{d,k} \to 0$. Using the Single Ring Theorem \cite{guionnet2011single}, \cite{rudelson2014invertibility}, and \cite{guionnet2012support},  we prove that when $k = 1$ the constant $c_{d,1}$ tends to $1$ ``asymptotically'' in both the real and complex cases. Furthermore, in the complex case, we get equality. 

\subsection{Statement}
The \emph{empirical singular value decompositon}  of a $d\times d$ matrix $A$ is the probability measure $\frac{1}{d} \sum \delta_{\sigma_i(A)}$ where $\sigma_i$ denotes the $i$th largest singular value. 

In the complex case, we have the following. 
\begin{theorem}
\label{concentration}
    Let $\mu_\sigma$ be a compactly supported probability measure on $\R_+$ whose support contains more than one point. Let $(A_d)_{d\in \N}$ be a sequence of $d \times d$ complex matrices such that for each $d \in \N$ \begin{enumerate}
        \item[-] there exists $C > 1$ such that for $C^{-1} \le \|A_d^{-1}\|^{-1} \le \|A_d\| \le C$ for each $d$ and
        \item[-] The empirical singular value distribution of $A_d$ converges to $\mu_\sigma$ in the weak star topology. 
    \end{enumerate}   Let $$R_+ := \left( \int x^2 d\mu_\sigma(x) \right)^{1/2}$$. Then the following statements hold:
\begin{enumerate}
\item[\textit{(1)}] Let $(U_d)$ be a sequence of random matrices, where $U_d$ is uniformly distributed in compact group $K = \mathrm{SO}(d)$ or $\mathrm{SU}(d)$. Then the sequence of random variables $\rho(U_d A_d )$ converges to $R_+$ in probability, where $\rho$ denotes the spectral radius. 
\item[\textit{(2)}] Let $(v_d)$ be a sequence of independent points, where $v_d$ is uniformly distributed in tbe unit sphere $S^{2d-1}$. Then the sequence of random variables $\|A_d v_d\|$ converges to $R_+$ in probability.
\end{enumerate}
\end{theorem}
This is proven in Section \ref{hello}, where the first result uses the \emph{Single Ring Theorem} and sequel results \cite{guionnet2011single}, \cite{rudelson2014invertibility}, \cite{localsrt}, \cite{guionnet2012support} and the second uses concentration of measure, particularly Lévy's inequality.

\begin{corollary}[Asymptotic Dedieu--Shub over $\mathrm{SU}(d)$ and $\mathrm{SO}(d)$] \label{asymptoticcomplex}
Suppose the same hypothesis as Theorem \ref{concentration}. Let $K = \O(d)$ or $\U(d)$. Let $\mu_\sigma$ be a compactly supported probability measure on $\R$. Suppose $A_d$ is a sequence of $d \times d$ matrices, so that the empirical singular value distribution of $A_d$ converges  to $\mu_\sigma$. Then
\begin{enumerate}
    \item 
    \begin{align*}
        \lim_{d \to \infty} \int_{K} |\lambda_1 (UA_d)| \ d \nu(U) &= \lim_{d \to \infty}  \int_{S^{2d-1}}  \|A_dv\| \ d\mu(v) = R_+
    \end{align*} and
    
\item 
\begin{align*}
        \lim_{d \to \infty} \int_{
    K} \log |\lambda_1 (UA_d)| \ d \nu(U) &= \lim_{d \to \infty}  \int_{S^{2d-1}} \log \|A_dv\| \ d\mu(v)= R_+
    \end{align*}
\end{enumerate}
and both limits exist. 
\end{corollary}



\subsection{Acknowledgements}

I thank Jairo Bochi, Asaf Katz, and Amie Wilkinson, for helpful comments. We thank Ofer Zeitouni for pointing out that we could significantly improve our result from an earlier preprint.

\section{Asymptotic results: Proof of Theorem \ref{concentration}}
\label{hello}

\subsection{The Single Ring Theorem}

Given a matrix $A$, the \emph{empirical eigenvalue distribution} of $A$ is the probability measure
$$\Lambda_{A} = \frac{1}{d} \sum \limits_{i=0}^{d-1} \delta_{\lambda_i(A)},$$
and the \emph{empirical singular value distribution} of $A$ is the probability measure
$$\Sigma_A = \frac{1}{d} \sum \limits_{i=0}^{d-1} \delta_{\sigma_i(A)}.$$ 

Let $\cP(\C)$ denote the space of probability measures on the Borel $\sigma$-algebra $(\C, \cB(\C)).$ Denote the $\epsilon$-neighborhood of $C \subset \C$ as $C^\epsilon$.  Recall that the \emph{Lévy--Prokhorov distance} is $$d_{LP}(\alpha, \beta) = \inf \left\{\epsilon > 0: \forall C \in \cB(\C), \text{ we have } \begin{cases}
    \beta(C) \leq \alpha(C^\epsilon) + \epsilon \\
    \alpha(C) \leq \beta(C^\epsilon) + \epsilon.
\end{cases} \right\}.$$ A useful fact about the $d_{LP}(\cdot, \cdot)$ is that it induces the weak topology \cite{billingsley2013convergence}. This means, $\nu_d \to \mu$ weakly if and only if $\lim \limits_{d \to \infty} d_{LP}(\nu_d, \mu) = 0.$

The Single Ring Theorem of Guionnet--Krishnapur--Zeitouni is as follows. 

\begin{theorem}[\cite{guionnet2011single},\cite{rudelson2014invertibility},\cite{localsrt}]
    Let $\mu_\sigma$ be a probability measure on $\R^+$, supported on more than one point, with the property that there exists a constant $C > 0$ such that for every $x$ in the support of $\mu_\sigma$, $$C^{-1} < x < C.$$ 

    Let $A_d$ be a sequence of $d \times d$ matrices such that \begin{enumerate}
        \item[i.] For every $d$ and $k \in \{1,...,d\}$, $C^{-1} < \sigma_k(A_d) < C$ and 
        \item[ii.] $d_{LP}(\Sigma_{A_d}, \mu_\sigma) \to 0 \text{ as } d \to \infty$.
    \end{enumerate} 
    Let $U_d$ be a sequence of independent Haar distributed matrices inside of $\U(d)$ or $\O(d)$. 

    Then there exists a probability measure $\mu_{SR}$ such that $$\Lambda_{U_dA_d} \to \mu_{SR}$$ weakly in probability. In particular, for all $\epsilon, \delta > 0$, there exists $N > 0$ such that for all $d > N$, we have \begin{equation}
  \nu(U_d \in \U(d): d_{LP}(\Lambda_{U_d A_d }, \mu_{SR}) > \epsilon) < \delta. \label{condition}
 \end{equation}

    Furthermore, $\mu_{SR}$ is a probability measure on $\C$ whose support is an annulus with inner radius $R_-$ and outer radius $R_+$, where $R_-$ and $R_+$ are given by $$R_+ =\left( \lim_{d \to \infty} \int_0^{\infty} x^2 d\mu_\sigma(x) \right) ^{1/2}, \ \ \ R_- = \left( \lim_{d \to \infty} \int_0^{\infty} x^{-2} d\mu_\sigma(x) \right) ^{-1/2}.$$ In addition, the densities of $\mu_{SR}$ are smooth and rotationally invariant.

Finallly, the Single Ring Theorem holds over $\O(d)$, except the Single Ring is no longer rotationally invariant. 
\end{theorem}

The original Single Ring Theorem included two additional assumptions, one removed in \cite{rudelson2014invertibility} and the other removed in \cite{localsrt}.

In a follow up theorem, Guionnet and Zeitouni were able to establish support convergence in the unitary and orthogonal case.

\begin{theorem}[Guionnet and Zeitouni \cite{guionnet2012support}] \label{GZconvergence}
    Let $K = \O(d)$ or $\U(d)$. Let $$SR_\delta = \{z \in \C: |z| \in [R_- -\delta, R_+ + \delta]\}.$$
    For any $\epsilon \text{ and } \delta >0$, there exists $ N  >0$ so that for all $d \geq N$, the empirical eigenvalue distribution of $U_dA_d$ is contained in a $\delta$-neighborhood of the Single Ring with probability greater than $1-\epsilon$. This means that $$\nu_d\{ U_d \in K: \mathrm{supp}(\Lambda_{U_dA_d}) \subset SR_\delta\} > 1-\epsilon$$ or equivalently $$\nu_d\{ U_d \in K: \mathrm{supp}(\Lambda_{U_dA_d}) 
    \not \subset SR_\delta\} < \epsilon.$$ 
\end{theorem}

In fact, Guionnet and Zeitouni proved a stronger result about almost sure convergence, but we do not need it.

\subsection{Proof of Theorem \ref{concentration}}

Concentration of measure says, roughly, that Lipschitz functions on high dimensional spheres are essentially constant. 

\begin{theorem}[Lévy's inequality for the median]  For Lipschitz $f:S^{d-1} \to \R$ with Lipschitz constant $L$, we have 
    $$\P_d(x \in S^{d-1}: |f(x) - M| \geq \epsilon) \leq 2 \exp \left( - \frac{cd\epsilon^2}{L^2}\right)$$ for some constant $c$ and $M$ is the median satisfying $$\P_d(f(x) \geq M) = 1/2 \text{ and } \P_d(f(x) \leq M) = 1/2.$$
\end{theorem}

\begin{corollary}
    \label{levy}

     For Lipschitz $f:S^{d-1} \to \R$ with Lipschitz constant $L$, we have 
    $$\P_d(x \in S^{d-1}: |f(x) - \E f| \geq \epsilon) \leq 2 \exp \left( - \frac{cd\epsilon^2}{L^2}\right)$$ for some constant $c$ an where $$\E f = \int f(x) \ dx,$$ where $\P_d$ is the normalized surface area on $S^{d-1}.$
\end{corollary}

\begin{proof}
See \cite{ledoux2001concentration} or \cite{vershynin2018high} for both.

\end{proof}

\begin{proof}[Proof of Theorem \ref{concentration} \textit{(2)}] 
The proof of part \textit{(2)} of Theorem \ref{concentration} uses Levy's inequality for concentration near the mean.  Assume $A_d = \mathrm{diag}(a_1,...,a_d)$. The assumption of the theorem states that each $a_i$ is bounded above and below by some fixed $C,1/C$, for every $d$. Define $$f_d: v \in S^{2d-1} \to \|A_d v\|^2.$$ The maps $v \to f_d(v)$ is $C$--Lipschitz and so $v \to \|Av\|$ is $2C^2$ Lipschitz, for all $d.$

   The symmetry of the sphere implies  $\E(v_1^2) = \E(v_2^2) = ... = \E(v_d^2)$. Because expectation is linear, we have $\E(v_1^2) = 1/d$.  Therefore, the expectation of $f_d$ is 
\begin{align*}
    \E f_d &= \int_{S^{d-1}} \|Av\|^2 \ dv = \int_{S^{d-1}} \sum \limits_{i=1}^d a_i^2v_i^2 \ dv = \sum \limits_{i=1}^d a_i^2 \int_{S^{d-1}} v_i^2 \ dv = \frac{1}{d} \sum \limits_{i=1}^d a_i^2. 
\end{align*}

In particular, $$\lim \limits_{d \to \infty} \E f_d = R_+^2,$$ where $R_+ = \sqrt{\int_{\R}x^2 \ d\mu_\sigma(x)}.$

Lévy's inequality asserts that $$\P_d[ |f_d(v) - \E f_d \geq \epsilon ] =  \P_d[ |\|A_dv\|^2 - R_+^2 | \geq \epsilon ]  \leq 2 \exp \left(- \frac{cd\epsilon^2}{L^2}\right). $$
In the limit, we have then that for all $\epsilon > 0$, then 
\begin{align*}
    \lim \limits_{d \to \infty}\P_d[ | \|A_d v\|^2 - R_+^2| \geq \epsilon ] &= \lim \limits_{d \to \infty} \P_d[ | \|A_d v\| - R_+| \cdot | \|A_d v \| + R_+|  \geq \epsilon ] \\
    & \leq \lim \limits_{d \to \infty}\ \P_d \left[|\|A_dv - R_+\| \geq \frac{\epsilon}{C + R_+} \right]\\ 
    & \leq \lim \limits_{d \to \infty} 2 \exp \left( - \frac{cd\epsilon^2}{L^2(C+R_+)^2} \right)\\
    &= 0.
\end{align*}

This implies that 
$$ \lim \limits_{d \to \infty}\P_d[ | \|A_d v\| - R_+\| \geq \epsilon ] = 0.$$

\end{proof}

As a corollary, we get the following.

\begin{corollary} \label{quadratic_mean}
    Let $\mu_\sigma$ be compactly supported. Let $A_d$ be a sequence of matrices such that there exists $C > 0$ so that $$C^{-1} < \|A_d^{-1}\|^{-1} < \|A_d\| < C,$$ where $\|\cdot\|$ is the $\sup$ norm, or the top singular value of $A$. Suppose $\Sigma_{A_d}$ converges weakly to a measure $\mu_\sigma$. We have that \begin{align*}
        \lim \limits_{d \to \infty} \int_{S^{d-1}} \|A_dv\| dv &= \lim \limits_{d \to \infty}  \sqrt{\int_{S^{d-1}} \|A_dv\|^2 dv } \\
        &= \sqrt{\int x^2 \ d\mu_\sigma(x)} = R_+.
    \end{align*}
    Furthermore \begin{align*}
        \lim \limits_{d \to \infty} \int_{S^{d-1}} \log \|A_dv\| dv = \log R_+.
    \end{align*}
\end{corollary}

\begin{proof}[Proof of Theorem \ref{concentration} (1)]

We wish to show for all $\epsilon > 0$, $$\lim_{d \to \infty} \nu_d \{ U_d \in K: |\rho(U_dA_d) - R_+| > \epsilon \} = 0.$$ This is equivalent to showing that $$\lim_{d \to \infty} \nu_d\underbrace{\{U_d: \rho(U_dA_d) < R_+ - \epsilon\}}_{(i)} + \lim \limits_{d \to \infty} \nu_d\underbrace{\{U_d: \rho(U_dA_d) > R_+ + \epsilon\}}_{(ii)} = 0.$$ 

We consider term $(i)$ first. Fix $\epsilon > 0$. We will argue that if $\rho(U_dA_d) < R_+ - \epsilon$ then necessarily $d_{LP}(\Lambda_{U_dA_d},\mu_{SR}) > \epsilon$. This will produce the contradiction as the Single Ring Theorem asserts this does not happen often. 

To this end, let $$S = \text{the annulus of inner radius } R_+ - \epsilon/2, \text{ and outer radius } R_+.$$ We know $\mu_{SR}(S) =: \epsilon' > 0 $ because $\mathrm{supp}(\mu_{SR}) \supseteq  S \supseteq  \mathrm{int}(S) \neq \emptyset.$ Now suppose for a given $d$, we have that for $U \in K$, $\rho(UA_d) < R_+ - \epsilon.$ Then $\Lambda_{UA_d} \subset \text{Disk}(0, R_+ - \epsilon).$ For this $U$, we have \begin{align*}
    d_{LP}(\Lambda_{UA_d}, \mu_{SR})  &= \inf \left\{\epsilon > 0: \forall C \in \cB(\C), \text{ we have } \begin{cases}
    \mu_{SR}(C) \leq \Lambda_{UA_d} (C^\epsilon) + \epsilon \\
    \Lambda_{UA_d}(C) \leq  \mu_{SR}(C^\epsilon) + \epsilon.
\end{cases} \right\} \\
&\geq \min \{\epsilon/2, \epsilon'\} \text{ replacing } C \text{ in the previous line with } S \\
&=: \epsilon''.
\end{align*}
We know by the Single Ring Theorem, that for all $\epsilon'' > 0$, 
$$\nu_d \{U_d \in K: d_{LP}(\lambda_{U_dA_d},\mu_{SR}) > \epsilon'' \} \to 0 \text{ as } d \to \infty.$$ So necessarily $$\nu_d \{U_d \in K: \rho(U_dA_d) < R_+ - \epsilon \} \to 0 \text{ as } d \to \infty.$$ 

We now treat term $(ii)$ using \cite{guionnet2012support}. Suppose for a given $d$, we have that for $U \in K$, $\rho(UA_d) > R_+ + \epsilon.$ In particular, the support of $\Lambda_{UA_d}$ is not strictly inside an $\epsilon$ neighborhood of the support of the Single Ring. But Theorem \ref{GZconvergence} implies that this does not happen often -- namely for $\delta >0$, there is a large $N$ such that for $d > N$, $$\nu_d\{ U_d \in K: \mathrm{supp}(\Lambda_{U_dA_d}) \not \subset SR_\delta\} = \epsilon.$$ Hence term $(ii)$ also decays to zero.
\end{proof}


\printbibliography

\end{document}